\documentclass[11pt,bezier,amstex,a4wide]{article}
\usepackage{latexsym,epsfig,amssymb,amsmath,amsthm,color,url,bbm,psfrag}
\usepackage{graphicx}
\usepackage{fullpage}
\allowdisplaybreaks
\numberwithin{equation}{section}
\allowdisplaybreaks[1]

\newcommand{\pr}{\rightarrow}

\newcommand{\ba}{\begin{array}}
\newcommand{\ea}{\end{array}}

\newcommand{\vart}{\vartheta}

\newcommand{\eps}{\varepsilon}

\newcommand{\il}{\int\limits}
\newenvironment{inspring}[1]%
{\begin{list}{}{\setlength{\rightmargin}{0cm}
                \setlength{\listparindent}{0cm}
                \settowidth{\labelwidth}{\mbox{#1}}
                \setlength{\leftmargin}{1.1\labelwidth}
                \setlength{\labelsep}{.1\labelwidth}}}%
{\end{list}}

\newcommand{\bi}[1]{\begin{inspring}{#1}}
\newcommand{\ei}{\end{inspring}}

\newcommand{\beq}{\begin{equation}}
\newcommand{\eq}{\end{equation}}
\catcode`\@=11

\font\tenmsa=msam10 \font\sevenmsa=msam7 \font\fivemsa=msam5
\font\tenmsb=msbm10 \font\sevenmsb=msbm7 \font\fivemsb=msbm5
\newfam\msafam
\newfam\msbfam
\textfont\msafam=\tenmsa  \scriptfont\msafam=\sevenmsa
  \scriptscriptfont\msafam=\fivemsa
\textfont\msbfam=\tenmsb  \scriptfont\msbfam=\sevenmsb
  \scriptscriptfont\msbfam=\fivemsb

\def\Bbb{\ifmmode\let\next\Bbb@\else
 \def\next{\errmessage{Use \string\Bbb\space only in math mode}}\fi\next}
\def\Bbb@#1{{\Bbb@@{#1}}}
\def\Bbb@@#1{\fam\msbfam#1}
\newcommand{\dR}{{\Bbb R}}

\newcommand{\dC}{{\Bbb C}}
\newcommand{\dZ}{{\Bbb Z}}

\newtheorem{thm}{Theorem}
\newtheorem{lem}[thm]{Lemma}

\newtheorem{prop}[thm]{Proposition}
\numberwithin{thm}{section}

\newcommand{\e}{{\rm e}}
\newcommand{\dom}{Z_{0}}
\newcommand{\rdom}{c_0}

\begin{document}
\title{Dominant poles and tail asymptotics in the \\critical Gaussian many-sources regime
\thanks{This work was financially supported by The Netherlands Organization for Scientific Research (NWO) and by an ERC Starting Grant.}
}
\author{
A.J.E.M. Janssen
\and
J.S.H. van Leeuwaarden
}

\maketitle
\footnotetext[1]{Eindhoven University of Technology, Department of Mathematics and Computer Science, P.O. Box 513, 5600 MB Eindhoven, The Netherlands.
\{a.j.e.m.janssen,j.s.h.v.leeuwaarden\}@tue.nl.

}

\begin{abstract}
The dominant pole approximation (DPA) is a classical analytic method to obtain from a generating function asymptotic estimates for its underlying coefficients.
We apply DPA to a discrete queue in a critical many-sources regime, in order to obtain tail asymptotics for the stationary queue length. As it turns out, this regime leads to a clustering of the poles of the generating function, which renders the classical DPA useless, since the dominant pole is not sufficiently dominant. To resolve this, we design a new DPA method, which might also find application in other areas of mathematics, like combinatorics, particularly when Gaussian scalings related to the central limit theorem are involved.
\end{abstract}

\section{Introduction}
Probability generating functions (PGFs) encode the distributions of discrete random variables. When PGFs are considered analytic objects, their singularities or poles contain crucial information about the underlying distributions. Asymptotic expressions for the tail distributions, related to large-deviations events, can typically be obtained in terms of the so-called dominant singularities, or dominant poles. The dominant pole approximation (DPA) for the tail distribution is then derived from the partial fraction expansion of the PGF and maintaining of this expansion the dominant fraction related to the dominant pole. Dominant pole approximations have been applied in many branches of mathematics, including analytic combinatorics \cite{flajolet} and queueing theory \cite{VanMieghem1996}. We apply DPA to a discrete queue that has an explicit expression for the PGF of the stationary queue length. Additionally, this queue is considered in a many-sources regime, a heavy-traffic regime in which both the demand on and the capacity of the systems grow large, while their ratio approaches one. This many-sources regime combines high system utilization and short delays, due to economies of scale. The regime is similar in flavor as the QED (quality and efficiency driven) regime for many-server systems \cite{halfinwhitt}, although an important difference is that our discrete queue fed by many sources falls in the class of single-server systems and therefore leads to a manageable closed form expression for the PGF of the stationary queue length $Q$. Denote this PGF by $Q(z)=\mathbb{E}(z^Q)$.

PGFs can be represented  as power series around $z=0$ with nonnegative coefficients (related to the probabilities). We assume that the radius of convergence of $Q(z)$ is larger than one (in which case all moments of $Q$ exist). This radius of convergence is in fact determined by the dominant singularity $\dom$, the singularity in $|z|>1$ closest to the origin. For PGFs, due to Pringsheim's theorem \cite{flajolet}, $\dom$ is always a positive real number larger than one. Then DPA leads to the approximation
\beq\label{tttijms}
\mathbb{P}(Q> N)\approx \frac{\rdom}{1-\dom}\Big(\frac{1}{\dom}\Big)^{N+1} \quad {\rm for} \ {\rm large} \ N
\eq
with $\rdom=\lim_{z\to \dom}(z-\dom)Q(z)$. In many cases the approximation \eqref{tttijms} can be turned into a more rigorous asymptotic expansion (for $N$ large) for the tail probabilities $\mathbb{P}(Q> N)$. We shall now explain in more detail the many-sources regime, the discrete queue, and when combining both, the mathematical challenges that arise when applying DPA.

 \vspace{.5cm}
\noindent {\bf Many sources and a discrete queue.} Consider a stochastic system in which demand per period is given by some random variable $A$, with mean $\mu_A$ and variance $\sigma^2_A$. For systems facing large demand one can set the capacity according to the rule $s=\mu_A+\beta \sigma_A$, which consists of a minimally required part $\mu_A$ and a variability hedge $\beta \sigma_A$. Such a rule can lead to economies of scale, as we will now describe in terms of a setting in which the demand per period is generated by many sources. Consider a system serving $n$ independent sources and let $X$ denote the generic random variable that describes the demand per source per period, with mean $\mu$ and variance $\sigma^2$. Denote the service capacity by $s_n$, so that the system utilization is given by $\rho_n=n\mu/s_n$, where the index $n$ expresses the dependence on the scale at which the system operates. The traditional capacity sizing rule would then be
\beq\label{aa}
s_n=n\mu+\beta \sigma \sqrt{n}
\eq
with $\beta$ some positive constant. The standard heavy-traffic paradigm \cite{britt1,nd1,nd2}, which builds on the Central Limit Theorem, then prescribes to consider a sequence of systems indexed by $n$ with associated loads $\rho_n$ such that (also using that $s=s_n\sim n\mu$)
\beq\label{bb}
\rho_n=\frac{n\mu}{s_n}\sim 1-\frac{\beta\sigma}{\mu\sqrt{n}}=1-\frac{\gamma}{\sqrt{s_n}}, \quad {\rm as} \ n\to\infty,
\eq
where $\gamma=\beta\sigma/\sqrt{\mu}$.
We shall apply the many-sources regime given by \eqref{aa} and \eqref{bb} to a discrete queue, in which we divide time into periods of equal length, and model the net input in consecutive periods as i.d.d.~samples from the distribution of $A$, with mean $n\mu$ and variance $n\sigma^2$. The capacity per period $s_n$ is fixed and integer valued.
 The scaling rule in \eqref{bb} thus specifies how the mean and variance of the demand per period, and simultaneously $s_n$, will all grow to infinity as functions of $n$. Many-sources scaling became popular through the Anick-Mitra-Sondhi model \cite{Anick1982}, as one of the canonical models for modern telecommunications networks, in which a switch may have hundreds of different input flows. But apart from communication networks, the concept of many sources can apply to any service system in which demand can be regarded as coming from many different inputs (see e.g.~\cite{Bruneel1993,johanthesis,Dai2014,Newell1960,vanLeeuwaarden2006} for specific applications.

\vspace{.5cm}
\noindent{\bf How to adapt classical DPA?}
As it turns out, the many-sources regime changes drastically the nature of the DPA. While the queue is pushed into the many-sources regime for letting $n\to\infty$, the dominant pole becomes barely dominant, in the sense that all the other poles (the dominated ones) of the PGF are approaching the dominant pole. For the partial fraction expansion of the PGF this means that it becomes hard, or impossible even, to simply discard the contributions of the fractions corresponding to what we call {\it dominated} poles: all poles other than the dominant pole. Moreover, the dominant pole itself approaches $1$ according to
\beq\label{dombeh}
\dom\sim 1+\frac{2\beta}{\sqrt{n\sigma^2}}, \quad {\rm as} \ n\to\infty.
\eq
This implies that in \eqref{tttijms} the factor $\rdom/(1-\dom)$ potentially explodes, while without imposing further conditions on $N$, the factor $\dom^{-N-1}$  goes to the degenerate value 1. The many-sources regime thus has a fascinating effect on the location of the poles that renders a standard DPA useless for multiple reasons. We shall therefore adapt the DPA in order to make it suitable to deal with the complications that arise in the many-sources regime, with the goal to again obtain an asymptotic expansion for the tail distribution. First observe that the term $\dom^{-N-1}$ in \eqref{tttijms} becomes non-degenerate when we impose that $N\sim K\sqrt{n\sigma^2}$, with $K$ some positive constant, in which case
\beq
\Big(\frac{1}{\dom}\Big)^{N+1}\sim \Big(1+\frac{2\beta}{\sqrt{n\sigma^2}}\Big)^{-K \sqrt{n\sigma^2}}\to \e^{-2\beta K}\in(0,1)
 \quad {\rm as} \  n\to\infty.
\eq
The condition $N\sim K\sqrt{n\sigma^2}$ is natural, because the fluctuations of our stochastic system are of the order $\sqrt{n\sigma^2}$. Of course, there are many ways in which $N$ and $n$ can be coupled, but due to \eqref{dombeh}, only couplings for which $N$ is proportional to $\sqrt{n}$ lead to a nondegenerate limit for $\dom^{-N-1}$. Now let us turn to the other two remaining issues: The fact that $\rdom/(1-\dom)$ potentially explodes and that the dominated poles converge to the dominant pole.

To resolve these two issues we present in this paper an approach that relies on approximations of the type \eqref{dombeh} for all the poles (which are defined implicitly as the solutions to some equation). The approximations are accurate in the many-sources regime, and can then be substituted into the partial fraction expansion that describes the tail distribution. We replace the partial fraction expansion by a contour integral representation, and subsequently apply a dedicated saddle point method recently introduced in \cite{britt1}, with again a prominent role for the dominant pole (this time in relation to the saddle point).
The key challenge is to bound the contributions of the contour integral when shifted beyond the dominant pole, a contribution which is substantial due to the relative large impact of the dominated poles. This saddle point method then provides a fully rigorous derivation of the asymptotic expression for $\mathbb{P}(Q> N)$  and is of the form
\beq \label{onesix}
\mathbb{P}(Q> K \sqrt{n\sigma^2})\sim h(\beta)\cdot \e^{-2\beta K}, \quad {\rm as} \ n\to\infty.
\eq
The function $h(\beta)$ in this asymptotic expression involves infinite series and Riemann zeta functions that are reminiscent of the reflected Gaussian random walk \cite{changperes,jllerch,cumulants}. Indeed, it follows from \cite[Theorem 3]{nd2} that our rescaled discrete queue converges under \eqref{bb} to a reflected Gaussian random walk. Hence, the tail distribution of our system in the  regime \eqref{bb} should for large $n$ be well approximated by the tail distribution of the reflected Gaussian random walk. We return to this connection in Subsection \ref{subsec4.3}.

Our approach thus relies on detailed knowledge about the distribution of all the poles of the PGF of $Q$, and in particular how this distribution scales with the asymptotic regime \eqref{aa}--\eqref{bb}. As it turns out, in contrast with classical DPA, this many-sources regime makes that all poles contribute to the asymptotic characterization of the tail behavior. Our saddle point method leads to an asymptotic expansion for the tail probabilities, of which the limiting form corresponds to the heavy-traffic limit, and pre-limit forms present refined approximations for pre-limit systems ($n<\infty$) in heavy traffic. Such refinements to heavy-traffic limits are commonly referred to as {\em corrected diffusion approximations} \cite{siegmund,blanchetglynn,asmussen}. Compared with the studies that directly analyzed the Gaussian random walk \cite{changperes,jllerch,cumulants}, which is the scaling limit of our queue in the many-sources regime, we start from the pre-limit process description, and establish an asymptotic result which is valuable for a queue with a finite yet large number of sources. Starting this asymptotic analysis from the actual pre-limit process description is mathematically more challenging than directly analyzing the process limit, but in return gives valuable insights into the manner and speed at which the system starts displaying its limiting behavior.
%

%

\vspace{.5cm}
\noindent{\bf Outline of the paper.}
In Section~\ref{sec2} we describe the discrete queue in more detail and present some preliminary results for its stationary queue length distribution. In Section~\ref{sec3} we give an overview of the results and the contour integration representation for the tail distribution. In Section~\ref{sec4}, we give a rigorous proof of the main result of the leading-order term  using the dedicated saddle point method (Subsection~\ref{subsec4.1}), and of bounding the contour integral with integration paths shifted beyond the dominant pole (Subsection~\ref{subsec4.2}). In Section~\ref{subsec4.3} we elaborate on the connection between the discrete queue and the Gaussian random walk, and we present an asymptotic series for $\mathbb{P}(Q >N)$ comprising not only the dominant poles but also the dominated poles. 

%

\section{Model description and preliminaries}\label{sec2}
We  consider a discrete stochastic model in which time is divided into periods of equal length. At the beginning of each period $k=1,2,3,...$ new demand $A_k$ arrives to the system. The demands per period $A_1,A_2,...$ are assumed independent and equal in distribution to some non-negative integer-valued random variable $A$.
The system has a service capacity $s\in\mathbb{N}$ per period, so that the recursion
\beq
\label{lind}
Q_{k+1} = \max\{Q_k + A_k - s,0\},\qquad k=1,2,...,
\eq
assuming $Q_1=0$, gives rise to a Markov chain $(Q_k)_{k\geq 1}$ that describes the congestion in the system over time. The PGF
\beq \label{e2}
A(z)=\sum_{j=0}^{\infty} \mathbb{P}(A=j) z^j
\eq
is assumed analytic in a disk $|z|<r$ with $r>1$, which implies that all moments of $A$ exist.

We assume that
 $A_k$ is in distribution equal to the sum of work generated by $n$ sources, $X_{1,k}+...+X_{n,k}$, where the $X_{i,k}$ are for all $i$ and $k$ i.i.d.~copies of a random variable $X$, of which the PGF $X(z)=\sum_{j=0}^{\infty}\mathbb{P}(X=j)z^j$ has radius of convergence $r>1$, and
\beq \label{e3}
0<\mu_A=n\mu=n X'(1)<s.
\eq




Under the assumption (\ref{e3}) the stationary distribution $\lim_{k\to\infty}\mathbb{P}(Q_k=j)=\mathbb{P}(Q=j)$, $j=0,1,\ldots$ exists, with the random variable $Q$ defined as having this stationary distribution.
We let
\beq \label{e4}
Q(z)=\sum_{j=0}^{\infty}\mathbb{P}(Q=j)z^j
\eq
be the PGF of the stationary distribution.

It is a well-known consequence of Rouch\'e's theorem that under \eqref{e3} $z^s-A(z)$ has precisely $s$ zeros in $|z|\leq1$, one of them being $z_0=1$. We proceed in this paper under the same assumptions as in \cite{britt1}. We assume that $|X(z)|<X(r_1)$, $|z|=r_1$, $z\neq r_1$, for any $r_1\in(0,r)$. Finally, we assume that the degree of $X(z)$ is larger than $s/n$. Under these conditions, $z_0$ is the only zero of $z^s-A(z)$ on $|z|=1$, and all others in $|z|\leq1$, denoted as $z_1,z_2,...,z_{s-1}$, lie in $|z|<1$. Furthermore, there are at most countably many zeros $Z_k$ of $z^s-A(z)$ in $1<|z|<r$, and there is precisely one, denoted by $Z_0$, with minimum modulus.

There is the product form representation \cite{Bruneel1993,johanthesis}
\beq \label{2.1}
Q(z)=\frac{(s-\mu_A)(z-1)}{z^s-A(z)} \prod_{j=1}^{s-1} \frac{z-z_j}{1-z_j},
\eq
where the right-hand side of (\ref{2.1}) is analytic in $|z|<Z_0$ and has a first-order pole at $z=Z_0$. We have for the tail probability (using that $Q(1)=1$) for $N=0,1,...$
\beq \label{2.2}
P(Q>N)=\sum_{i=N+1}^{\infty} P(Q=i)=C_{z^N} \Bigl[\frac{1-Q(z)}{1-z}\Bigr],
\eq
where $C_{z^N}[f(z)]$ denotes the coefficient of $z^N$ of the function $f(z)$. By contour integration, Cauchy's theorem and $Q(1)=1$, we then get for $0<\eps<Z_0-1$
\begin{eqnarray} \label{2.3}
P(Q>N) & = & \frac{1}{2\pi i} \il_{|z|=1+\eps} \frac{1}{z^{N+1}}~\frac{1-Q(z)}{1-z} dz \nonumber \\[3.5mm]
& = & \frac{1}{2\pi i} \il_{|z|=R} \frac{1}{z^{N+1}}~\frac{1-Q(z)}{1-z} dz+\frac{c_0}{Z_0^{N+1}(1-Z_0)},
\end{eqnarray}
where $c_0={\rm Res}_{z=Z_0}[Q(z)]$ and $R$ is any number between $Z_0$ and $\min_{k\neq0}|Z_k|$. When $n$ and $s$ are fixed, we have that the integral on the second line of (\ref{2.3}) is $O(R^{-N})$, and so there is the DPA
\beq \label{2.4}
P(Q>N)=\frac{c_0}{Z_0^{N+1}(1-Z_0)} (1+\mbox{exponentially small}),~~~~N\pr\infty.
\eq

In this paper we crucially rely on
Pollaczek's integral representation for the PGF of Q
\beq \label{2.5}
Q(z)=\exp\Big(\frac{1}{2\pi i} \il_{|v|=1+\eps} \ln \Bigl(\frac{z-v}{1-v}\Bigr) \frac{(v^s-A(v))'}{v^s-A(v)} dv\Big)
\eq
that holds when $|z|<1+\eps<Z_0$ (principal value of ln on $|v|=1+\eps$).

%

\section{Overview and results} \label{sec3}
In order to force the discrete queue to operate in the critical many-sources regimes, we shall assume throughout the paper the following relation between the number of sources $n$ and the capacity $s$:
\beq \label{1.2}
\frac{n\mu}{s}=1-\frac{\gamma}{\sqrt{s}}
\eq
with $\gamma>0$ bounded away from 0 and $\infty$ as $s\pr\infty$.
In this scaling regime, the zeros $z_j$ and $Z_k$ of $z^s-A(z)=0$ start clustering near $z=1$, as described in the next lemma (proved in the appendix). Let $z_j^*$ and $Z_k^*$ denote the complex conjugates of $z_j$ and $Z_k$, respectively.

\begin{lem}\label{lemdis} For finite $j,k=1,2,...$ and $s\pr\infty$,
\beq \label{3.1}
z_0=1,~~~~~~Z_0=1+\frac{2a_0b_0}{\sqrt{s}}+O(s^{-1}),
\eq
\beq \label{3.2}
z_j=z_{s-j}^{\ast}=1+\frac{a_0}{\sqrt{s}} (b_0-\sqrt{b_0^2-2\pi ij})+O(s^{-1}),
\eq
\beq \label{3.3}
Z_k=Z_{-k}^{\ast}=1+\frac{a_0}{\sqrt{s}} (b_0+\sqrt{b_0^2-2\pi ik})+O(s^{-1})
\eq
with
\beq \label{3.4}
a_0=\frac{\sqrt{2\mu}}{\sigma},~~~~~~b_0=\frac{\gamma\sqrt{\mu}}{\sigma\sqrt{2}}
\eq
and principal roots in \eqref{3.2}-\eqref{3.4}.
\end{lem}

Due to this clustering phenomenon, the main reasoning that underpin classical DPA cannot be carried over. Starting from the
expression \eqref{2.4} we need to investigate what becomes of the term $c_0/(1-Z_0)$, and moreover, the validity of the exponentially-small phrase in (\ref{2.4}) and the actual $N$-range both become delicate matters that need detailed information about the distribution of the zeros as in Lemma \ref{lemdis}.

Let us first present a result that identifies the relevant $N$-range:
\begin{prop} \label{prop4.6}
\beq \label{4.42}
\frac{1}{Z_0^{N+1}}=\exp\Bigl(\frac{-2L\gamma\mu}{\sigma^2}\Bigr)(1+O(s^{-1/2}))
\eq
when $N+1=L\sqrt{s}$ with $L>0$ bounded away from 0 and $\infty$.
\end{prop}
\begin{proof} We have from (\ref{3.1}) and \eqref{3.4} that $Z_0=1+\frac{2\gamma\mu}{\sigma^2\sqrt{s}}+O(\frac{1}{s})$. Hence
\begin{eqnarray} \label{4.43}
\frac{1}{Z_0^{N+1}}  =  \exp\Bigl({-}L\sqrt{s} \ln \Bigl( 1+\frac{2\gamma\mu}{\sigma^2\sqrt{s}}+O(s^{-1})\Bigr)\Bigr) = \exp\Bigl(\frac{-L\gamma\mu}{\sigma^2}+O(s^{-1/2})\Bigr)
\end{eqnarray}
when $L$ is bounded away from 0 and $\infty$, and this gives the result.
\end{proof}

From (\ref{2.1}) we obtain the representation
\beq \label{3.5}
\frac{c_0}{1-Z_0}={-} \frac{s-\mu_A}{sZ_0^{s-1}-A'(Z_0)} \prod_{j=1}^{s-1} \frac{Z_0-z_j}{1-z_j}.
\eq
The next result will be proved in Section~\ref{sec4}.

\begin{lem} \label{lem4.1}
\beq \label{4.2}
- \frac{s-\mu_A}{sZ_0^{s-1}-A'(Z_0)}=\Big(\frac{1}{Z_0}\Big)^{s-1}\Big(1+O(s^{-1/2})\Big).
\eq
\end{lem}

We thus get from (\ref{3.5}) and Lemma~\ref{lem4.1}
\beq \label{4.10}
\frac{c_0}{1-Z_0}=\frac{P(Z_0)}{P(1)} \Bigl(1+O(s^{-1/2})\Bigr),
\eq
where
\beq \label{4.11}
P(Z)=\prod_{j=1}^{s-1} (1-z_j/Z)=\exp\Bigl(\sum_{j=1}^{s-1} \ln (1-z_j/Z)\Bigr)
\eq
for $Z\in\dC$, $|Z|\geq1$ (principal logarithm).
To handle the product $P(z)$, in Lemma~\ref{lem4.3} below, we evaluate $\ln  P(Z)$ for $|Z|\geq1$ in terms of the contour integral
\beq \label{4.12}
I(Z)=\frac{1}{2\pi i} \il_{|z|=1+\eps} \frac{\ln (1-z^{-s}A(z))}{Z-z} dz,
\eq
where $\eps>0$ is such that $1<1+\eps<Z_0$.
\begin{lem} \label{lem4.3}
Let $\eps>0$, $1<1+\eps<Z_0$ and $|Z|\geq1$. Then
\beq \label{4.13}
\ln  P(Z)=\left\{\ba{lll}
- \ln (1-Z^{-1})+I(Z), \ 1+\eps < |Z|< r,\\
- \ln (1-Z^{-1})+\ln (1-Z^{-s}A(Z))+I(Z), \ 1<|Z|<1+\eps,\\
~~\:\ln (\gamma\sqrt{s})+I(1), \ Z=1.
\ea\right.
\eq
\end{lem}

The dedicated saddle point method, as considered in \cite{britt1}, applied to $I(Z)$, with saddle point $z_{\rm sp}=1+\eps$ of the function $g(z)={-}\ln  z+\frac{n}{s} \ln (X(Z))$, yields
\beq \label{3.9}
I(1)={-}\ln  [Q(0)]+O(s^{-1/2}),\quad I(Z_0)=\ln  [Q(0)]+O(s^{-1/2}).
\eq
Combining (\ref{3.1}), (\ref{3.4}), (\ref{3.5}), (\ref{4.2}) and (\ref{3.9}), then gives one of our main results:
\begin{prop} \label{prop4.4}
\beq \label{4.16}
\ln \Bigl(\frac{c_0}{1-Z_0}\Bigr)={-}\ln (4b_0^2)+2\ln  [Q(0)]+O(s^{-1/2}).
\eq
\end{prop}

The next step consists of bounding the integral on the second line of (\ref{2.3}), that can be written as
\beq \label{3.12}
\frac{-1}{2\pi i} \il_{|z|=R} \frac{Q(z)}{z^{N+1}(1-z)} dz,
\eq
by choosing $R$ appropriately. To do this, we consider the product representation (\ref{2.1}) of $Q(z)$, and we want to choose $R$ such that $|z^s-A(z)|\geq C |z|^s$, $|z|=R$, for some $C>0$ independent of $s$. It will be shown in Section~\ref{sec4} that this is achieved by taking $R$ such that the curve $|z^s|=|A(z)|$, on which $Z_0$ and $Z_{\pm1}$ lie, is crossed near a point $z$ (also referred to as $Z_{\pm1/2}$), where $z^s$ and $A(z)$ have opposite sign. A further analysis, using again the dedicated saddle point method to bound the product $\prod_{j-1}^{s-1}$ in (\ref{2.1}), then yields that the integral in (\ref{3.12}) decays as $R^{-N}$. Finally using the asymptotic information in (\ref{3.1})-(\ref{3.3}) for $Z_0$ and $Z_{\pm1}$, with $Z_{\pm1/2}$ lying midway between $Z_0$ and $Z_{\pm1}$, the integral on the second line of (\ref{2.3}) can be shown to have relative order $\exp({-}DN/\sqrt{s})$, for some  $D>0$ independent of $s$, compared to the dominant-pole term in (\ref{2.4}).

To summarize, we have now that
\beq\label{317}
\mathbb{P}(Q> N)= \frac{\rdom}{1-\dom}\Big(\frac{1}{\dom}\Big)^{N+1} \Big(1+O(\e^{-DN/\sqrt{s}})\Big),
\eq
for some $D>0$ independent of $s$, $N=1,2,\ldots$. The DPA $\rdom/(1-\dom)^{-1}\dom^{-N-1}$ of $\mathbb{P}(Q> N)$ thus has a relative error that decays exponentially fast.

In Subsection~\ref{subGRW} the stationary queue length $Q$, considered in the many-sources regime, is shown to be connected to the Gaussian random walk. This connection will imply that the front factor of the DPA in \eqref{317} satisfies
\beq\label{318}
\frac{\rdom}{1-\dom}=H(b_0) \Big(1+O(s^{-1/2})\Big),
\eq
where $\ln H(b_0)$ has a power series in $b_0$ with coefficients that can be expressed in terms of the Riemann zeta function. Combining this with Proposition \ref{prop4.6} and \eqref{317} yields
\beq\label{319}
\mathbb{P}(Q> N)=H(b_0)\exp\Big(-\frac{2L\gamma\mu}{\sigma^2}\Big)\Big(1+O(s^{-1/2})\Big)
\eq
when $N+1=L\sqrt{s}$ with $L$ bounded away from $0$ and $\infty$. The leading term in \eqref{319} agrees with \eqref{onesix} when we identify
\beq\label{320}
L=\sigma K/\sqrt{\mu}, \quad \gamma=\beta\sigma/\sqrt{\mu}, \quad s=n\mu+\beta\sigma\sqrt{n}\approx n\mu, \quad b_0=\beta/\sqrt{2}
\eq
and $H(b_0)=h(\beta)$.

In Subsection~\ref{subsubsec4.3.2} we extend for a fixed $M=1,2,\ldots$ the approach in Section \ref{sec4} by increasing the radius $R$ of the integration contour in \eqref{3.12} to $R_M$ such that the poles $Z_0, Z_{\pm 1},\ldots,Z_{\pm M}$ are inside $|z|=R_M$. this lead to
\begin{eqnarray} \label{321}
P(Q>N)={\rm Re} \Bigl[\frac{c_0}{(1-Z_0) Z_0^{N+1}}  +  2 \sum_{k=1}^M \frac{c_k}{(1-Z_k) Z_k^{N+1}}\Bigr] +O\Big(|Z_{M+1}|^{-N}\Big).
\end{eqnarray}
The front factors $c_k/(1-Z_k)$ in the series in \eqref{321} satisfy
\beq\label{322}
\frac{c_k}{1-Z_k}=H_k(b_0) \Big(1+O(s^{-1/2})\Big),
\eq
with $H_k(b_0)$ some explicitly defined integral. When $N+1=L\sqrt{s}$ with $L$ bounded away from $0$ and $\infty$, we find from
\eqref{322} and Proposition \ref{prop4.6} that
\beq\label{323}
\frac{c_k}{1-Z_k}=H(b_k)\exp\Big(-L a_0(\sqrt{b_0^2-2\pi i k}+b_0)\Big)\Big(1+O(s^{-1/2})\Big),
\eq
compare with \eqref{318}, and it can be shown that this gives rise to an $\exp(-DL/\sqrt{k})$ decay of the right-hand side of \eqref{323}.
The results in \eqref{319} and \eqref{323} together give precise information as to how the DPA arises, with leading behavior from the dominant pole, and lower order refinements coming from the dominated poles.

\section{DPA through contour integration} \label{sec4}

In this section we present the details of getting approximations of the tail probabilities using a contour integration approach as outlined in Section~\ref{sec3}. In Subsection~\ref{subsec4.1}, we concentrate on approximation of the front factor $c_0/(1-Z_0)$ and the dominant pole $Z_0$, and combine these to obtain an approximation of the leading-order term in (\ref{2.4}). This gives Lemma \ref{lem4.1}, Lemma \ref{lem4.3} and Proposition \ref{prop4.4}.

In Subsection~\ref{subsec4.2} we assess and bound the integral on the second line of (\ref{2.3}) and thereby make precise what exponentially small in (\ref{2.4}) means in the present setting.

\subsection{Approximation of the leading-order term} \label{subsec4.1}

\subsubsection{Proof of Lemma \ref{lem4.1}}


From
\beq \label{4.3}
Z_0^s=A(Z_0)=X^n(Z_0) ,~~\mu_A=n \mu=s\Bigl(1-\frac{\gamma}{\sqrt{s}}\Bigr) ,~~A'(Z_0)=n X'(Z_0) X^{n-1}(Z_0),
\eq
we compute
\beq \label{4.4}
\frac{s-\mu_A}{sZ_0^{s-1}-A'(Z_0)}=\frac{\gamma/\sqrt{s}}{1-(1-\gamma/\sqrt{s}) \dfrac{X'(Z_0)Z_0}{X'(1) X(Z_0)}}~\frac{1}{Z_0^{s-1}}.
\eq
With the approximation \eqref{3.1}, written as
\beq \label{4.5}
Z_0=1+\frac{d_0}{\sqrt{s}}+O(s^{-1}),~~~~~~d_0=\frac{2\gamma\mu}{\sigma^2},
\eq
 we get
\begin{eqnarray} \label{4.6}
X'(Z_0) =  X'(1)+X''(1)(Z_0-1)+O(s^{-1})  = \mu+\frac{X''(1) d_0}{\sqrt{s}}+O(s^{-1})
\end{eqnarray}
and
\begin{eqnarray} \label{4.7}
X(Z_0)  =  X(1)+X'(1)(Z_0-1)+O(s^{-1}) = 1+\frac{\mu d_0}{\sqrt{s}}+O(s^{-1}).
\end{eqnarray}
Hence, by (\ref{4.5}--\ref{4.7}) and (\ref{A6}),
\begin{eqnarray} \label{4.8}
& \mbox{} & 1-\Bigl(1-\frac{\gamma}{\sqrt{s}}\Bigr) \frac{X'(Z_0) Z_0}{X'(1) X(Z_0)} \nonumber \\[3mm]
& & =~1-\Bigl(1-\frac{\gamma}{\sqrt{s}}\Bigr) \frac{\Bigl(\mu+\dfrac{X''(1) d_0}{\sqrt{s}}+O(s^{-1})\Bigr) \Bigl(1+\dfrac{d_0}{\sqrt{s}}+O(s^{-1})\Bigr)}{\mu\Bigl(1+\dfrac{\mu d_0}{\sqrt{s}}+O(s^{-1})\Bigr)} \nonumber \\[3mm]
& & =~1-\bigl(1-\frac{\gamma}{\sqrt{s}}\Bigr)\Bigl(1+\frac{d_0}{\sqrt{s}} \Bigl(\frac{X''(1)}{\mu}+1-\mu\Bigr)+O(s^{-1})\Bigr) \nonumber \\[3mm]
& & =~1-\bigl(1-\frac{\gamma}{\sqrt{s}}\Bigr)\Bigl(1+\frac{d_0}{\sqrt{s}}~\frac{\sigma^2}{\mu}+O(s^{-1})\Bigr)={-}  \frac{\gamma}{\sqrt{s}}+O(s^{-1}),
\end{eqnarray}
where we have used $d_0$ of (\ref{4.5}) in the last step. This gives (\ref{4.2}).

\subsubsection{Proof of Lemma \ref{lem4.3}}
We have $|A(z)|<|z^s|$ when $z\neq1$, $1<|z|<Z_0$, and so $\ln (1-A(z)/z^s)$ is analytic in $z\neq1$, $1<|z|<Z_0$. When $|Z|>1+\eps$, we have by partial integration and Cauchy's theorem
\begin{eqnarray} \label{4.14}
I(Z) & = & \frac{1}{2\pi i} \il_{|z|=1+\eps} \ln \Bigl(1-\frac{z}{Z}\Bigr) \frac{(1-A(z)/z^s)'}{1-A(z)/z^s} dz \nonumber \\[3.5mm]
& = & \sum_{j=0}^{s-1} \ln \Bigl(1-\frac{z_j}{Z}\Bigr)=\ln \bigl(1-\frac1Z\Bigr)+\ln  P(Z).
\end{eqnarray}
This gives the upper-case formula in (\ref{4.13}), and the middle case follows in a similar manner by taking the residue at $z=Z$ inside $|z|=1+\eps$ into account. For the lower case in (\ref{4.13}), we use the result of the middle case, in which we take $1<Z<1+\eps$, $Z\downarrow1$. We have $I(Z)\pr I(1)$ as $Z\downarrow1$, and
\begin{eqnarray} \label{4.15}
& \mbox{} & \lim_{Z\downarrow1} \Bigl[{-}\ln \Bigl(1-\frac1Z\Bigr)+\ln \Bigl(1-\frac{X^n(Z)}{Z^s}\Bigr)\Bigr] \nonumber \\[3mm]
& & =~\ln  [(Z^s-X^n(Z))'|_{Z=1}]=\ln (s-n\mu)=\ln (\gamma\sqrt{s}),
\end{eqnarray}
and this completes the proof.
\subsubsection{Proof of Proposition \ref{prop4.4}}
By (\ref{4.10}) and Lemma~\ref{lem4.3} we have
\begin{eqnarray} \label{4.17}
\hspace*{-1cm}\ln \Bigl(\frac{c_0}{1-Z_0}\Bigr) & = & \ln  P(Z_0)-\ln  P(1)+O(s^{-1/2}) \nonumber \\[3mm]
& = & {-}\ln \Bigl(1-\frac{1}{Z_0}\Bigr)-\ln (\gamma\sqrt{s})+I(Z_0)-I(1)+O(s^{-1/2}) .
\end{eqnarray}
From \eqref{3.1} it follows that
\beq \label{4.18}
-\ln \Bigl(1-\frac{1}{Z_0}\Bigr)-\ln (\gamma\sqrt{s})={-}\ln \bigl(\frac{2\gamma^2 \mu}{\sigma^2}\Bigr)+O(s^{-1/2}) ={-}\ln (4b_0^2)+O(s^{-1/2}),
\eq
and so
\beq \label{4.19}
\ln \Bigl(\frac{c_0}{1-Z_0}\Bigr)=I(Z_0)-I(1)-\ln (4b_0^2)+O(s^{-1/2}).
\eq
Next, we consider the integral representation (\ref{4.12}) of $I(Z)$, where we take $\eps$ such that
\beq \label{4.20}
1+\eps=z_{\rm sp}=1+\frac{\gamma\mu}{\sigma^2\sqrt{s}}+O(s^{-1}),
\eq
with $z_{\rm sp}$, see \cite[Section~3]{britt1}, the unique point $z\in(1,Z_0)$ such that
\beq \label{4.21}
\frac{d}{dz} \Bigl[{-}\ln  z+\frac{n}{s} \ln (X(z))\Bigr]=0.
\eq
Observe that
\beq \label{4.22}
z_{\rm sp}=\tfrac12 (1+Z_0)+O(s^{-1}),~~~~~~Z_0-z_{\rm sp}=z_{\rm sp}-1+O(s^{-1}),
\eq
and this suggests that $I(Z_0)\approx{-}I(1)$, a statement made precise below, since the main contribution to $I(Z)$ comes from the $z$'s in (\ref{4.12}) close to $z_{\rm sp}$. \\
We have, see \cite{britt1},
\beq \label{4.23}
\ln  [P(Q=0)]=\frac{1}{2\pi i} \il_{|z|=z_{\rm sp}} \frac{\ln \Bigl(1-z^{-s}A(z)\Bigr)}{z(z-1)} dz.
\eq
Now
\beq \label{4.24}
\frac{1}{z-1}=\frac{1}{z(z-1)}+\frac1z,
\eq
and
\beq \label{4.25}
\il_{|z|=z_{\rm sp}} \Bigl|\ln \Bigl(1-z^{-s}A(z)\Bigr)\Bigr| |dz|=O(s^{-1/2}),
\eq
see  \cite[Subsection~5.3]{britt1}. Hence,
\beq \label{4.26}
I(1)=\frac{-1}{2\pi i} \il_{|z|=z_{\rm sp}} \frac{\ln \Bigl(1-z^{-s}A(z)\Bigr)}{z-1} dz=\ln  [P(Q=0)]+O(s^{-1/2}).
\eq
As to $I(Z_0)$, we observe that, see (\ref{4.22}),
\beq \label{4.27}
\frac{1}{Z_0-z}=\frac{1}{z-1}+2 \frac{z-\frac12 (1+Z_0)}{(Z_0-z)(z-1)}=\frac{1}{z-1}+2  \frac{z-z_{\rm sp}}{(Z_0-z)(z-1)}+O(1).
\eq
Thus,
\beq \label{4.28}
I(Z_0)={-}I(1)+2 \il_{|z|=z_{\rm sp}} \frac{z-z_{\rm sp}}{(Z_0-z)(z-1)} \ln \Bigl(1-z^{-s}A(z)\Bigr) dz+O(s^{-1/2}).
\eq
We next estimate the remaining integral in (\ref{4.28}). With the substitution $z=z(v)$, $-\frac12 \delta\leq v\leq\frac12 \delta$, we have $A(z(v))/z^s(v)=B \exp({-}s\eta v^2)$ with $0<B<1$ and $\eta>0$ bounded away from 1 and 0, respectively, and
\beq \label{4.29}
z(v)=z_{\rm sp}+iv+\sum_{n=2}^{\infty} c_n(iv)^n,~~~~~~-\tfrac12 \delta\leq v\leq\tfrac12 \delta,
\eq
where $c_n$ are real. Then we get with exponentially small error
\begin{eqnarray} \label{4.30}
\il_{|z|=z_{\rm sp}} \frac{z-z_{\rm sp}}{(Z_0-z)(z-1)} \ln \Bigl(1-z^{-s}A(z)\Bigr) dz =\il_{-\frac12\delta}^{\frac12\delta} \frac{(z(v)-z_{\rm sp}) z'(v)}{(Z_0-z(v))(z(v)-1)} \ln (1-B e^{-s\eta v^2}) dv.
\end{eqnarray}
Now we get from (\ref{4.22}) and (\ref{4.29}) that
\begin{eqnarray} \label{4.31}
(Z_0-z(v))(z(v)-1) & = & (z_{\rm sp}-1-iv+O(v^2))(z_{\rm sp}-1+iv+O(v^2)) \nonumber \\[3mm]
& & +~O \Bigl(s^{-1} (z_{\rm sp}-1+iv+O(v^2))\Bigr) \nonumber \\[3mm]
& = & |z_{\rm sp}-1|^2+v^2+O \Bigl(\Bigl(s^{-1}+v^2\Bigr)^{3/2}\Bigr).
\end{eqnarray}
Furthermore,
\beq \label{4.32}
(z(v)-z_{\rm sp}) z'(v)={-}v+O(v^2).
\eq
Thus
\beq \label{4.33}
\frac{(z(v)-z_{\rm sp}) z'(v)}{(Z_0-z(v))(z(v)-1)}=\frac{-v+O(v^2)}{|z_{\rm sp}-1|^2+v^2+O\Bigl(\Bigl(s^{-1}+v^2\Bigr)^{3/2}\Bigr)}.
\eq
Inserting this into the integral on the second line of (\ref{4.30}), we see that the $-v$ in (\ref{4.33}) cancels upon integration. Also
\beq \label{4.34}
\il_{-\frac12\delta}^{\frac12\delta} \frac{v^2}{|z_{\rm sp}-1|^2+v^2} \ln (1-B e^{-\frac12 s\eta v^2}) dv=O(s^{-1/2}),
\eq
and this finally shows that the integral in (\ref{4.30}) is $O(s^{-1/2})$. Then combining (\ref{4.19}), (\ref{4.26}), (\ref{4.28}), we get the result.

\subsection{Bounding the remaining integral} \label{subsec4.2}

We have from (\ref{2.3})
\beq \label{4.44}
P(Q>N)=\frac{c_0}{Z_0^{N+1}(1-Z_0)}-\frac{1}{2\pi i} \il_{|z|=R} \frac{Q(z)}{z^{N+1}(1-z)}  dz,
\eq
where $R\in(Z_0,|Z_{\pm1}|)$, and we intend to bound the integral at the right-hand side of (\ref{4.44}). We use in (\ref{4.44}) the $Q(z)$ as represented by the right-hand side of (\ref{2.1}) which is defined and analytic in $z$, $|z|<r$, $z\neq Z_k$. We write for $|z|<r$, $z\neq Z_k$
\beq \label{4.45}
\frac{Q(z)}{(1-z) z^{N+1}}=\frac{-1}{z^{N+2-s}}~\frac{s-\mu_A}{\prod_{j=1}^{s-1} (1-z_j)} ~\frac{1}{z^s-A(z)} \prod_{j=1}^{s-1} (1-z_j/z).
\eq
Now $s-\mu_A=\gamma\sqrt{s}$, and by Lemma~\ref{lem4.3} and (\ref{4.26}), we have $\prod_{j=1}^{s-1} (1-z_j)=P(1)\geq C \gamma\sqrt{s}$ for some $C>0$ independent of $s$. Hence $(s-\mu_A)/\prod_{j=1}^{s-1} (1-z_j)$ is bounded in $s$. Next, for $|z|\geq Z_0$, we have by Lemma~\ref{lem4.3}
\beq \label{4.46}
\prod_{j=1}^{s-1} (1-z_j/z)=\frac{z}{z-1} \exp (I(z)),
\eq
with $I(z)$ given by (\ref{4.12}) and admitting an estimate
\beq \label{4.47}
|I(z)|=O \Big(|z-z_{\rm sp}|^{-1} \il_{-\infty}^{\infty} \ln (1-B e^{-\frac12 s\eta t^2}) dt\Big)=O(1)
\eq
since $\sqrt{s}|z-z_{\rm sp}|$, $B\in(0,1)$ and $\eta>0$ are all bounded away from 0. Therefore, there remains to be considered $(z^s-A(z))^{-1}$. We show below that there is a $C>0$, independent of $s$, such that
\beq \label{4.48}
|z^s-A(z)|\geq C |z|^s
\eq
when $z$ is on a contour $K$ as in Figure \label{fig1}, consisting of a straight line segment
\beq \label{4.49}
z=\xi+i\eta,~~~~~~\xi={\rm Re} [\hat{Z}({\pm}\tfrac12)],~~~~~~{-} \frac{1}{\sqrt{s}} y_0\leq\eta\leq\frac{1}{\sqrt{s}} y_0,
\eq
and a portion of the circle
\beq \label{4.50}
|z|=R=\sqrt{{\rm Re}^2 [\hat{Z}({\pm}\tfrac12)]+\dfrac{1}{s} y_0^2}
\eq
that are joined at the points $({\rm Re}[\hat Z(\pm \tfrac12)],{\pm}\frac{1}{\sqrt{s}} y_0)$. Here
\beq \label{4.51}
\hat{Z}(t)=1+\frac{a_0}{\sqrt{s}} ((b_0^2-2\pi it)^{1/2}+b_0),
\eq
with $a_0,b_0>0$ given in (\ref{A10}) and independent of $s$, approximates the solution $z=Z(t)$, for real $t$ small compared to $s$, of the equation
\beq \label{4.52}
\frac{n}{s} \ln  X(z)-\ln  z=\frac{2\pi it}{s}
\eq
outside the unit disk, according to
\beq \label{4.53}
Z(t)=\hat{Z}(t)+O\Bigl(\frac{t}{s}\Bigr).
\eq
Thus on $K$ we have from (\ref{4.48})
\beq \label{4.54}
\Bigl|\frac{Q(z)}{(1-z) z^{N+1}}\Bigr|=O\Bigl(\frac{1}{(z-1) z^{N+1}}\Bigr),
\eq
and we estimate
\begin{align} \label{4.55}
& \Big|\frac{1}{2\pi i} \il_{|z|=R} \frac{Q(z)}{z^{N+1}(1-z)} dz\Big|  =  \Big|\frac{1}{2\pi i} \il_{z\in K} \frac{Q(z)}{z^{N+1}(1-z)} dz\Big| \nonumber \\
& =  O \Big(\Bigl(\frac{1}{{\rm Re} [\hat{Z}({\pm}\frac12)]}\Bigr)^{N+1} \il_{z\in K} \frac{|dz|}{|z-1|}\Big) = O \Big(\ln  s\Bigl(\frac{1}{{\rm Re} [\hat{Z}({\pm}\frac12)]}\Bigr)^{N+1}\Big).
\end{align}
Here we have used that $|z-1|\geq{\rm Re} [\hat{Z}({\pm}\tfrac12)]-1\geq E/\sqrt{s}$, $z\in K$, for some $E>0$ independent of $s$. Observing that
\beq \label{4.56}
\hat{Z}(0)=1+\frac{2a_0b_0}{\sqrt{s}},
\eq
\beq \label{4.57}
{\rm Re} [\hat{Z}({\pm}\tfrac12)]=1+\frac{a_0}{\sqrt{s}} [(\tfrac12 b_0^2+\tfrac12(b_0^4+\pi^2 )^{1/2})+b_0],
\eq
we see that
\begin{eqnarray} \label{4.58}
\bigg(\frac{\hat{Z}(0)}{{\rm Re} [\hat{Z}({\pm}\frac12)]}\bigg)^N & = & \Bigg(1-\frac{\dfrac{a_0}{\sqrt{s}} [(\frac12 b_0^2+\frac12(b_0^4+\pi^2)^{1/2})^{1/2}-b_0]} {1+\dfrac{a_0}{\sqrt{s}} [(\frac12 b_0^2+\frac12(b_0^4+\pi^2)^{1/2})^{1/2}+b_0]} \Bigg)^N \nonumber \\[3.5mm]
& = & O(\exp({-}\hat{D}N/\sqrt{s}))
\end{eqnarray}
for some $\hat{D}>0$ independent of $s$. Hence, by (\ref{4.53}), we see that the relative error in (\ref{4.44}) due to ignoring the integral at the right-hand side is of order $\exp({-}DN/\sqrt{s})$ with some $D>0$, independent of $s$.

We show the inequality (\ref{4.48}) for $z\in K$ using the following property of $X$: there is a $\delta>0$ and a $\vart_1\in(0,\pi/2)$ such that for any $R\in[1,1+\delta]$ the function $|X(R e^{i\vart})|$ is decreasing in $|\vart|\in[0,\vart_1]$ while
\beq \label{4.59}
\vart_1\leq|\vart|\leq\pi\Rightarrow |X(R e^{i\vart})|\leq |X(R e^{i\vart_1})|.
\eq
This property follows from strict maximality of $|X(e^{i\vart})|$ in $\vart\in[{-}\pi,\pi]$ at $\vart=0$ and analyticity of $X(z)$ in the disk $|z|<r$ (with $r>1$).

For the construction of the contour $K$ in (\ref{4.48}--\ref{4.49}), we consider the quantity
\beq \label{4.60}
n \ln  [X(1+v)]-\ln (1+v),
\eq
where $v$ is of the form
\beq \label{4.61}
v=\frac{2\gamma\mu}{\sigma^2\sqrt{s}}+\frac{x_0+iy}{\sqrt{s}}=\hat{Z}(0)-1+\frac{x_0+iy}{\sqrt{s}}
\eq
with $x_0>0$ fixed and varying $y\in\dR$. We choose $x_0$ such that the outer curve $Z(t)$ is crossed by $z=1+v$ near the points $Z({\pm}\frac12)$, where $z^s-A(z)$ equals $2z^s$. Thus, we choose
\beq \label{4.62}
x_0={\rm Re} [\sqrt{s} (\hat{Z}({\pm}\tfrac12)-\hat{Z}_0)]={\rm Re} [a_0((b_0^2+\pi i)^{1/2}-b_0)].
\eq
We have, as in the analysis in the appendix, that
\begin{eqnarray} \label{4.63}
n \ln  [X(1+v)]-s \ln (1+v) =\frac{2\gamma\mu}{\sigma^2} x_0+(x_0^2-y^2)+2i\Bigl(\frac{\gamma\mu}{\sigma^2}-x_0\Bigr) y+O(sv^3)+O(v^2\sqrt{s}) .
\end{eqnarray}
With $x_0>0$ fixed and independent of $s$, see (\ref{4.62}), the leading part of the right-hand side in (\ref{4.63}) is independent of $s$ and describes, as a function of the real variable $y$, a parabola in the complex plane with real part bounded from above by its real value at $y=0$ and that passes the imaginary axis at the points $\pm\pi i$. Therefore, this leading part has a positive distance to all points $2\pi ik$, integer $k$. Now take $y_0$ such that
\beq \label{4.64}
\frac{2\gamma\mu}{\sigma^2}+(x_0^2-y_0^2)={-} \Bigl(\frac{2\gamma\mu}{\sigma^2}+x_0^2\Bigr).
\eq
In Figure \ref{fig1} we show the curve $K$ (heavy), the approximation $\hat Z(t)$ of the outer curve, and the choice $y_0=\eta_0\sqrt{s}$ for the case that $\gamma=1$, $\mu/\sigma^2=2$, $s=100$.

It follows from the above analysis, with $v$ as in (\ref{4.61}), that
\beq \label{4.65}
1-\frac{X^n(1+v)}{(1+v)^s},~~~~~~-y_0\leq y\leq y_0,
\eq
is bounded away from 0 and has a value $1-c,1-c^{\ast}$ at $y={\pm} y_0$, where $c$ is bounded away from 1 and $|c|<1$. Now write
\beq \label{4.66}
R e^{i\vart_0}=\hat{Z}(0)+\frac{x_0+iy_0}{\sqrt{s}}=1+v_0.
\eq
When $s$ is large enough, we have that $R\in[1,1+\delta]$ and $0\leq\vart_0\leq\vart_1$, where $\delta$ and $\vart_1$ are as above in (\ref{4.59}). We have
\beq \label{4.67}
|X^n(1+v_0)|\leq |c| \ |1+v_0)|^s,
\eq
and by (\ref{4.59}) and monotonicity of $|X(R e^{i\vart})|$, $\vart_0\leq|\vart|\leq\vart_1$,
\beq \label{4.68}
|X^n(R e^{i\vart})|\leq |X^n(R e^{i\vart_0})|\leq |c| R^s,~~~~~~\vart_0\leq|\vart|\leq\pi.
\eq
Therefore, (\ref{4.48}) holds on $K$ with
\beq \label{4.68a}
C=\min \Bigl\{1-|c| , \min_{|y|\leq y_0} \Bigl| 1-\frac{X^n(1+v)}{(1+v)^s}\Bigr|\Bigr\}
\eq
positive and bounded away from 0 as $s$ gets large.

\begin{figure}
\begin{center}
 \includegraphics[width= .5 \linewidth, angle =90 ]{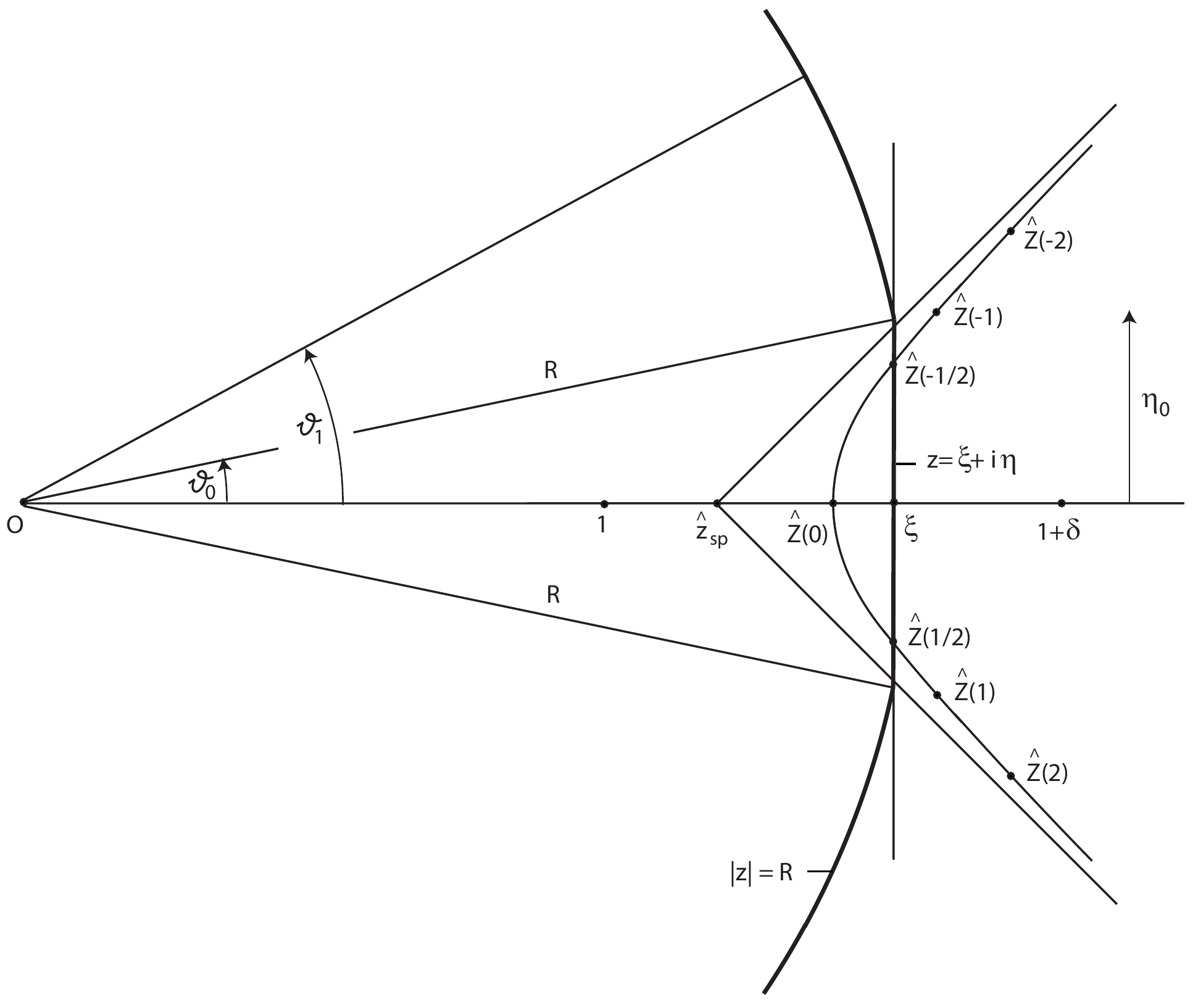}
 \end{center}
  \caption{Integration curve $K$ consisting of line segment $z=\xi+i\eta$, $-\eta_0\leq\eta\leq \eta_0$, where $\xi={\rm Re}[\hat Z(\pm \tfrac12)]$ and $\eta_0=y_0/\sqrt{s}$, and portion of the circle $|z|=R$ with $R=(\xi^2+\eta_0^2)^{1/2}$. Choice of parameters: $\gamma=1$, $\mu/\sigma^2=2$, $s=100$.
  }
  \label{fig1}
\end{figure}

\section{Correction terms and asymptotic expansion} \label{subsec4.3}

In this section we give a series expansion for the leading term in \eqref{4.16}
involving the Riemann zeta function. 
  We also show how to find an asymptotic series for $P(Q>N)$ as $N\pr\infty$ of which the term involving the dominant pole is the leading term. Before we do so, we first discuss how this leading term is related to the Gaussian random walk and a result of Chang and Peres \cite{changperes}.

\subsection{Connection with Gaussian random walk}\label{subGRW}
We know from  \cite[Theorem 3]{nd2} that under the critical many-sources scaling, the rescaled queueing process converges to a reflected Gaussian random walk. The latter is defined as
 $(S_\beta(k))_{k\geq 0}$ with $S_\beta(0)=0$ and
\begin{equation}
S_\beta(k)=Y_1+\ldots+Y_k
\end{equation}
with $Y_1,Y_2,\ldots$ i.i.d.~copies of a normal random variable with mean $-\beta$ and variance 1.
Assume $\beta>0$ (negative drift), and denote the all-time maximum of this random walk by ${M}_\beta$.

Denote by $Q^{(s)}$ the stationary congestion level for a fixed $s$ (that arises from taking
 $k\to \infty$ in \eqref{lind}).
Then, using $\rho_s=1-\gamma/\sqrt{s}$, with
 \begin{equation}\label{gammachoice}
 \gamma=\frac{\beta\sigma\sqrt{s}}{\mu\sqrt{n}},
\end{equation}
the spatially-scaled stationary queue length reaches the limit
$Q^{(s)}/(\sigma\sqrt{n}) \stackrel{d}{\to} {M}_\beta$ as $s,n\to\infty$ (see \cite{jelenkovic,nd1,nd2}).

The random variable ${M}_\beta$ was studied in  \cite{changperes,jllerch,cumulants}. In particular, \cite[Thm.~1]{jllerch} yields, for $\beta<2\sqrt{\pi}$,

\beq \label{1.5}
\mathbb{P}({M}_\beta=0)=\sqrt{2} \beta \exp \Bigl\{\frac{\beta}{\sqrt{2\pi}} \sum_{r=0}^{\infty} \frac{\zeta(\frac12-r)({-}\frac12 \beta^2)^r} {r! (2r+1)}\Bigr\},
\eq
and from \cite{changperes} we have $\mathbb{P}({M}_\beta> K)=h(\beta,K) e^{-2\beta K}$ with
\beq \label{1.6}
h(\beta,K)\pr h(\beta)=\exp \Bigl\{\frac{\beta\sqrt{2}}{\sqrt{\pi}}\:\sum_{r=0}^{\infty} \frac{\zeta(\frac12-r)({-}\frac12 \beta^2)^r} {r! (2r+1)}\Bigr\},
\eq
exponentially fast as $K\pr\infty$.

Hence, there are the approximations
\beq \label{estimate}
\mathbb{P}(Q> K \sqrt{n\sigma^2})\approx \mathbb{P}({M}_\beta>K)\approx  h(\beta)\cdot \e^{-2\beta K}, \quad {\rm as} \ n\to\infty,
\eq
where the second approximation holds for small values of $\beta$. We will now show how this second approximation in \eqref{estimate} follows from our leading term in the expansion.

\begin{prop} \label{prop4.5}
\beq \label{4.35}
\ln \Bigl(\frac{c_0}{1-Z_0}\Bigr)=\frac{2b_0}{\sqrt{\pi}} \sum_{r=0}^{\infty}  \frac{\zeta(\frac12-r)({-}b_0^2)^r}{r! (2r+1)}+ O(s^{-1/2}),~~~~~~0<b_0<\sqrt{2\pi}.
\eq
\end{prop}
\begin{proof}It is shown in \cite{britt1}, Subsection~5.3 that
\beq \label{4.37}
\ln  [Q(0)]=\ln  [P(M_\beta=0)]+O(s^{-1/2}),
\eq
in which we take the drift parameter $\beta$ according to
\beq \label{4.38}
\beta=b_0\sqrt{2}=\gamma \sqrt{\dfrac{\mu}{\sigma^2}}.
\eq
From \cite{cumulants} we have
\beq \label{4.39}
\ln  [P(M_\beta=0)]=\ln (2b_0)+\frac{b_0}{\sqrt{\pi}} \sum_{r=0}^{\infty} \frac{\zeta(\frac12-r)({-}b_0^2)^r}{r! (2r+1)},~~~~~0<b_0<\sqrt{2\pi}.
\eq
Then from Proposition~\ref{prop4.4}, (\ref{4.37}) and (\ref{4.39}), we get the results in (\ref{4.35}).
\end{proof}

\subsection{Asymptotic series for $P(Q>N)$ as $N\pr\infty$} \label{subsubsec4.3.2}

When inspecting the argument that leads to (\ref{2.3}), it is obvious that one can increase the radius $R$ of the integration contour to values $R_M$ between $|Z({\pm}M)|$ and $|Z({\pm}(M+1))|$ when $M=1,2,...$ is fixed. Here it must be assumed that $s$ is so large that $Z_k$ increases in $k=0,1,...,M+1$. Then, the poles of $Q(z)$ at $z=Z_{\pm k}$, $k=0,1,...,M $, are inside $|z|=R_M$, and we get
\begin{eqnarray} \label{4.69}
P(Q>N)=\frac{c_0}{(1-Z_0) Z_0^{N+1}}  +  2 \sum_{k=1}^M {\rm Re} \Bigl[\frac{c_k}{(1-Z_k) Z_k^{N+1}}\Bigr] - \frac{1}{2\pi i} \il_{|z|=R_M} \frac{Q(z)}{(1-z) z^{N+1}} dz.
\end{eqnarray}
As in Subsection~\ref{subsec4.2}, one can argue that the integral on the second line of (\ref{4.69}) is relatively small compared to $|Z_M|^{-N-1}$ when $R_M$ is chosen between but away from $|Z_M|$ and $|Z_{M+1}|$.

We now need the following result.
\begin{lem} \label{cor4.2}
There holds
\beq \label{4.9}
{-} \frac{s-\mu_A}{sZ_k^{s-1}-A'(Z_k)}=\frac{b_0}{\sqrt{b_0^2-2\pi ik}}~\frac{1}{Z_k^{s-1}} \Bigl(1+O\Bigl(\frac{1+|k|}{\sqrt{s}}\Bigr)\Bigr)
\eq
when $k=o(s)$.
\end{lem}
\begin{proof}
This follows from the appendix with a similar argument as in the proof of Lemma~\ref{lem4.1}.
\end{proof}
As to the terms in the series in (\ref{4.69}), we have for bounded $k$, see Lemma~\ref{cor4.2},
\begin{eqnarray} \label{4.70}
\frac{c_k}{1-Z_k} & = & - \frac{s-\mu_A}{sZ_k^{s-1}-A'(Z_k)} \prod_{j=1}^{s-1} \frac{Z_k-z_j}{1-z_j} \nonumber \\[3.5mm]
& = & \frac{b_0}{\sqrt{b_0^2-2\pi ik}}\cdot\prod_{j=1}^{s-1} \frac{1-z_j/Z_k}{1-z_j}\cdot\Bigl( 1+O(s^{-1/2})\Bigr).
\end{eqnarray}
Furthermore, according to Lemma~\ref{lem4.3},
\begin{eqnarray} \label{4.71}
\ln  \Bigl[\prod_{j=1}^{s-1} \frac{1-z_j/Z_k}{1-z_j}\Bigr] & = & I(Z_k)-I(1)-\ln \Bigl(1-\frac{1}{Z_k}\Bigr)-\ln (\gamma\sqrt{s}) \nonumber \\[3.5mm]
& = & I(Z_k)-I(1)-\ln  [2b_0(b_0+\sqrt{b_0^2-2\pi ik})]+ O(s^{-1/2}).
\end{eqnarray}
Thus, we get the following result.
\begin{prop} \label{prop4.7}
For bounded $k\in\dZ$,
\beq \label{4.72}
\frac{c_k}{1-Z_k}=\frac{\exp(I(Z_k)-I(1))}{2\sqrt{b_0^2-2\pi ik}  (b_0+\sqrt{b_0^2-2\pi ik})} \Bigl(1+O(s^{-1/2})\Bigr).
\eq
\end{prop}

We aim at approximating $I(Z_k)$, showing, in particular, that $c_k/(1-Z_k)\neq0$ is bounded away from $0$ for bounded $k$ and large $s$. To that end, we conduct the dedicated saddle point analysis for $I(Z_k)$. We have for $|Z|\geq Z_0$, ${\rm Re}(Z)>z_{\rm sp}$,
\begin{eqnarray} \label{4.73}
I(Z) & = & \frac{1}{2\pi i} \il_{|z|=z_{\rm sp}}  \frac{\ln \Bigl(1-z^{-s}A(z)\Bigr)}{Z-z} dz \nonumber \\[3.5mm]
& = & \frac{1}{2\pi i} \il_{-\frac12\delta}^{\frac12\delta}  \frac{z(v)}{Z-z(v)} \ln (1-B e^{-\frac12 s\eta v^2}) dv,
\end{eqnarray}
with exponentially small error in the last identity as $s\pr\infty$.
With $g(z)={-}\ln  z+\frac{n}{s} \ln  X(z)$, we let
\beq \label{4.74}
B=\exp(sg(z_{\rm sp}))=e^{-b_0^2}\Bigl(1+O(s^{-1/2})\Bigr) ,~~~~\eta=g''(z_{\rm sp})=\frac{\sigma^2}{\mu}+O(s^{-1/2}),
\eq
and $z(v)$ is as in (\ref{4.29}) and defined implicitly by $g(z(v))=g(z_{\rm sp})-\frac12 v^2g''(z_{\rm sp})$. We then find, by using $z(v)=z_{\rm sp}+iv+O(v^2)$ and $z'(v)=i+O(v)$, that
\begin{eqnarray} \label{4.75}
I(Z) & = & \frac{-1}{2\pi i} \il_{-\infty}^{\infty} \frac{\ln (1-B e^{-\frac12 s\eta v^2})}{v+i(Z-z_{\rm sp})} dv+O(s^{-1/2}) \nonumber \\[3.5mm]
& = & \frac{-1}{2\pi i} \il_{-\infty}^{\infty} \frac{\ln (1-B e^{-t^2})}{t+i(Z-z_{\rm sp}) \sqrt{s\eta/2}} dt+ O(s^{-1/2}),
\end{eqnarray}
where in the last step the substitution $t=v \sqrt{s\eta/2}$ has been made. Combining in the last integrand in (\ref{4.75}) the values at $t$ and $-t$ for $t\geq0$, we get the following result.
\begin{prop} \label{prop4.8}
For $|Z|\geq Z_0$, ${\rm Re}(Z)>z_{\rm sp}$,
\beq \label{4.76}
I(Z)=J(d)+O(s^{-1/2}),
\eq
where $d=(Z-z_{\rm sp}) \sqrt{s\eta/2}$, and
\beq \label{4.77}
J(d)=\frac{1}{\pi} \il_0^{\infty} \frac{d}{t^2+d^2} \ln (1-B e^{-t^2}) dt.
\eq
\end{prop}

In the context of Proposition~\ref{prop4.7}, we consider
\beq \label{4.78}
d=d_k=(Z_k-z_{\rm sp}) \sqrt{ s\eta/2 },
\eq
with
\beq \label{4.79}
z_{\rm sp}=1+\frac{a_0b_0}{\sqrt{s}}+O(s^{-1}),~~~~~~Z_k=1+\frac{a_0}{\sqrt{s}}  (\sqrt{b_0^2-2\pi ik}+b_0)+O(s^{-1}).
\eq
Using the definitions of $a_0$, $b_0$ in (\ref{A10}) and $\eta$ in (\ref{4.74}), we get
\beq \label{4.80}
d_k=\hat{d}_k+O(s^{-1/2})~;~~~~~~\hat{d}_k=(b_0^2-2\pi ik)^{1/2}.
\eq
We have that
\beq \label{4.81}
|\hat{d}_k|\geq b_0,~~~~~~{\rm arg}(\hat{d}_k)\in({-}\tfrac14 \pi,\tfrac14 \pi), ~~~~~~k\in\dZ.
\eq
Since for $t\geq0$ and ${\rm arg}(d)\in({-}\frac14 \pi,\tfrac14 \pi)$, we have
\beq \label{4.82}
\ln (1-B e^{-t^2})<0,~~~~~~{\rm arg}\Bigl(\frac{d}{t^2+d^2}\Bigr)\in({-}\tfrac14 \pi,\tfrac14 \pi),
\eq
we see that we have complete control on the quantities $J(\hat{d}_k)$ (also note (\ref{4.74}) for this purpose). Using that $-I(1)=I(Z_0)+O(s^{-1/2})$, see (\ref{4.28}), we get the following result.
\begin{prop} \label{prop4.9}
For bounded $k\in\dZ$,
\beq \label{4.83}
\frac{c_k}{1-Z_k}=\frac{\exp(J(\hat{d}_k)+J(\hat{d}_0))}{2\hat{d}_k(\hat{d}_k+\hat{d}_0)} +O(s^{-1/2}),
\eq
where the leading quantity in (\ref{4.83}) $\neq 0,\infty$,  $\hat{d}_k$ is given in (\ref{4.80}) with $\hat{d}_0=b_0$, and $J$ is given in (\ref{4.77}).
\end{prop}
\begin{thm} \label{thm4.10}
There is the asymptotic series
\beq \label{4.84}
P(Q>N)\sim{\rm Re} \Bigl[\frac{c_0}{(1-Z_0) Z_0^{N+1}}+2 \sum_{k=1}^{\infty} \frac{c_k}{(1-Z_k)   Z_k^{N+1}}\Bigr],
\eq
where the ratio of the terms in the series with index $M$ and $M-1$ is $O(|Z_{M-1}/Z_M|^N)$.
\end{thm}
\begin{proof}
This follows from (\ref{4.69}), in which the integral is $o(|Z_M|^{-N})$ and the term with $k=M$ is $O(|Z_M|^{-N})$, while the reciprocal of the term with $k=M-1$ is $O(|Z_{M-1}|^{-N})$ by Proposition~\ref{prop4.9}. In the consideration of the terms with $k=M-1,M $, it is tacitly accumed that $s$ is so large that $|Z_k|$, $k=0,1,...,M$ is a strictly increasing sequence.
\end{proof}


\appendix
\section{Proof of Lemma \ref{lemdis}}
We consider the zeros $z_j$, $j=0,1,...,s-1 $, and $Z_k$, $k\in\dZ$, of the function $z^s-A(z)$ in the unit disk $|z|\leq1$ and in the annulus $1<|z|<r$, respectively, in particular those that are relatively close to 1. These zeros are elements of the set $S_{A,s}=\{z\in\dC | |z|<r,~|z^s|=|A(z)|\}$.
For  $z\in S_{A,s}$, we have that $\ln (z^s X^{-n}(z))$ is purely imaginary. We thus consider the equation
\beq \label{A1}
s \ln  z=n \ln  X(z)+2\pi it
\eq
with $z$ near 1 and $t$ small compared to $s$. Writing
\beq\ \label{A2}
u=2\pi t, ~~~~~~z=1+v,
\eq
we get by Taylor expansion around $z=1$ the equation
\beq \label{A3}
s(v-\tfrac12 v^2+O(v^3))=n \ln (1+X'(1)v+\tfrac12 X''(1)v^2+O(v^3))+iu.
\eq
Dividing by $s$ and using that $\frac{n}{s} X'(1)=1-\gamma/\sqrt{s}$, yields
\beq \label{A4}
v-\tfrac12 v^2+O(v^3)=\Bigl(1-\frac{\gamma}{\sqrt{s}}\Bigr) \Bigl(v+\frac{X''(1)-(X'(1))^2}{2X'(1)} v^2+O(v^3)\Bigr)+i \frac{u}{s},
\eq
i.e.,
\beq \label{A5}
- \frac{\gamma}{\sqrt{s}} v+\frac{\sigma^2}{2\mu} v^2+i \frac{u}{s}= O\Bigl(\frac{v^2}{\sqrt{s}}\Bigr)+O(v^3),
\eq
where we have used that
\beq \label{A6}
\mu=X'(1), ~~~~~~\sigma^2=X''(1)-(X'(1))^2+X'(1)>0.
\eq
Dividing in (\ref{A5}) by $\sigma^2/2\mu$ and completing a square, we get
\beq \label{A7}
\Bigl(v-\frac{\gamma\mu}{\sigma^2\sqrt{s}}\Bigr)^2=\Bigl(\frac{\gamma\mu}{\sigma^2\sqrt{s}}\Bigr)^2 \Bigl(1-\frac{2iu\sigma^2}{\gamma^2\mu}\Bigr)+O\Bigl(\frac{v^2}{\sqrt{s}}\Bigr)+O(v^3).
\eq
Taking square roots at either side of (\ref{A7}) and using that the leading term at the right-hand side of (\ref{A7}) has order $(1+|u|)/s$, we get
\beq \label{A8}
v=\frac{\gamma\mu}{\sigma^2\sqrt{s}}\pm\frac{\gamma\mu}{\sigma^2\sqrt{s}} \Bigl( 1-\frac{2iu\sigma^2}{\gamma^2\mu}\Bigr)^{1/2}+O\Bigl(\frac{|v|^2+|v|^3\sqrt{s}} {\sqrt{1+|u|}}\Bigr).
\eq
Irrespective of the $\pm$-sign, the leading part of right-hand side of (\ref{A8}) has order $((1+|u|)/s)^{1/2}$ (and even $(|u|/s)^{1/2}$ in the case of the $-$-sign), and the $O$-term has order $(1+|u|)/s$ which is $o(((1+|u|)/s)^{1/2})$ as long as $|u|/s=o(1)$. Thus, in that regime of $u$, we have
\begin{eqnarray} \label{A9}
z=1+v & = & 1+\frac{\gamma\mu}{\sigma^2\sqrt{s}} \Bigl(1\pm\Bigl( 1-\frac{2iu\sigma^2}{\gamma^2\mu}\Bigr)^{1/2}\Bigr)+O\Bigl(\frac{1+|u|}{s}\Bigr) \nonumber \\[3mm]
& = & 1\pm\frac{a_0}{\sqrt{s}} ((b_0^2-iu)^{1/2}\pm b_0)+O\Bigl(\frac{1+|u|}{s}\Bigr),
\end{eqnarray}
where we have inserted
\beq \label{A10}
a_0=\Bigl(\frac{2\mu}{\sigma^2}\Bigr)^{1/2}, ~~~~~~b_0=\Bigl(\frac {\gamma^2\mu}{2\sigma^2}\Bigr)^{1/2}=\tfrac12 \gamma a_0.
\eq
In the case of the minus sign in (\ref{A9}), the $O$-term may be replaced by $O(|u|/s)$.
Choosing $u=2\pi j$ with $j=0,1,...,$ and $j=o(s)$, we get from (\ref{A9}) with the minus sign, \eqref{3.2}.
Choosing $u=2\pi k$ with $k\in\dZ$ and $k=o(s)$, we get from (\ref{A9}) with the plus sign, \eqref{3.3}.

\bibliography{Bibliography}

\end{document}